\documentclass[12pt]{amsart}
\usepackage{amsfonts}
\usepackage{amssymb,amsthm,amscd,amstext,amsmath,mathrsfs,CJK,latexsym}
\usepackage{float}
\usepackage{tikz}
\usetikzlibrary{intersections,decorations.pathmorphing,decorations.fractals}

\newtheorem{lemma}{Lemma}

\newtheorem{corollary}{Corollary}

\newtheorem{example}{Example}

\newcommand{\C}{\mathbb{C}}

\begin{document}

\title[]
{A New Proof of $\bf{Z_{Kup}=|Z_{Henn}|^2}$ for Semisimple Hopf Algebras}

\author{Liang Chang}
\email{liangchang@math.tamu.edu}
\address{Department of Mathematics\\
    Texas A$\&$M University \\
    College Station, TX 77843-3368\\
    U.S.A.}

\maketitle

\begin{abstract}

Hennings and Kuperberg defined quantum invariants $Z_{Henn}$ and $Z_{Kup}$ for closed oriented $3$-manifolds based on certain Hopf algebras,
respectively. When the Hopf algebras are semisimple, it is shown that $Z_{Kup}=|Z_{Henn}|^2$. In this paper, we present a new proof of this equality.

\end{abstract}

\section{Introduction}

Since the interaction between low-dimensional topology and quantum physics in the 1980s, many quantum invariants of 3-manifolds have been studied. A type of quantum invariants $Z_{Henn}$ was defined by Hennings \cite{H} and reformulated by Kauffman and Radford \cite{KR} based on finite dimensional ribbon Hopf algebras. More generally, Kuperberg defined the invariants $Z_{Kup}$ for closed framed 3-manifolds (\cite{Ku}), which are based on any finite dimensional Hopf algebras. If the Hopf algebras are semisimple, then these two types of quantum invariants are equal to the Reshetikhin-Turaev invariant $Z_{RT}$ \cite{ReTu}  and the Turaev-Viro invariant $Z_{TV}$ \cite{TuVi1}, which are defined from the representation categories of the input Hopf algebras, see \cite{BW} and \cite{Kerler1}. In this sense, $Z_{Henn}$ and $Z_{Kup}$ can be considered as non-semisimple generalizations of $Z_{TV}$ and $Z_{RT}$.

It is conjectured that $Z_{Kup}=|Z_{Henn}|^2$ for any finite dimensional ribbon Hopf algebras and any closed 3-manifolds, which generalizes the equality $Z_{TV}=|Z_{RT}|^2$ (\cite{TuVi2}, \cite{BK}). This equality was verified to be true for lens spaces in \cite{ChaWa}. For a general closed 3-manifold it has been proven implicitly for semisimple Hopf algebra because $Z_{Kup}=Z_{TV}$ and $Z_{Henn}=Z_{RT}$ when we can turn to work on the representation categories of the given Hopf algebras. In this paper, we prove this equality by completely working on Hopf algebras and using the diagram presentation introduced in\cite{ChaWa}. As a result, we have

\vspace{.1in}
\noindent
{\bf Theorem.}
{\it Let $H$ be a finite dimensional semisimple factorizable  Hopf algebra and $M$ be an oriented closed 3-manifold.  Then $Z_{Kup}(M,f,H)=Z_{Henn}(M\#\overline{M},H)$ for some suitably chosen framing $f$ on M.}
\vspace{.1in}

As a corollary, we have $Z_{Kup}(M,f,H)=|Z_{Henn}(M,H)|^2$ if $H$ is equipped with certain anti-linear automorphism (see the corollary in section 3).

Another motivation to consider these two types of 3-manifold invariant is that when the 3-manifold is fixed, $Z_{Kup}$ and $Z_{Henn}$ provide gauge invariants for Hopf algebras. Two Hopf algebras are said to be gauge equivalent if their representation categories are equivalent as tensor categories. For example, the Frobenius-Schur indicators of Hopf algebras are gauge invariants (\cite{KMN}, \cite{KSY}), which have important application to the representation theory of Hopf algebras and turns out to coincide with $Z_{Kup}$ for lens space \cite{ChaWa}. It is expected that $Z_{Kup}$ provides generalized Frobenius-Schur indicators and more general gauge invariants for any finite dimensional Hopf algebra (not necessarily semisimple). This is true for certain ribbon Hopf algebras due to the recent result in \cite{CheKer}. That is, $Z_{Henn}$ is shown to be gauge invariant for any factorizable ribbon Hopf algebra (\cite{CheKer}) and is related to $Z_{Kup}$ by the results in this paper and \cite{ChaWa}. The diagram presentation in this paper is useful in studying the gauge invariance of $Z_{Kup}$ and the discussion will appear in the subsequent paper.  

The paper is organized as follows. In Section $2$, we recall the definitions of the Hennings and Kuperberg invariants and set up our notations. Then we prove our theorem in Section $3$.

\section{Hennings and Kuperberg invariants}

\subsection{Some facts about Hopf algebras}

In this section, we recall necessary notations and structures of finite dimensional semisimple Hopf algebras over $\mathbb{C}$. Details can be found in \cite{Ra1}, \cite{Ra2} and \cite{KR2}. 

Let $H(m, \Delta, S, 1, \epsilon)$ be a finite dimensional Hopf algebra over $\C$ with multiplication $m$, comultiplication $\Delta$, antipode $S$, unit $1$, and counit $\epsilon$. In the following, we will apply the Sweedler notion, i.e.,
$\Delta^{(n-1)}(x)=\sum\limits_{(x)}x_{(1)}\otimes\ldots\otimes x_{(n)}$.

For finite dimensional semisimple Hopf algebras, there exists $\lambda\in H^*$, called \textit{two sided integral} (\textit{integral}, for short), such that $$(id\otimes\lambda)\Delta(x)=(\lambda\otimes id)\Delta(x)=\lambda(x)\cdot1.$$ Dually, there exists $\Lambda\in H$ called \textit{two sided cointegral} (\textit{cointegral}, for short), such that $$x\Lambda=\Lambda x=\varepsilon(x)\Lambda.$$ Such elements are unique up to a scalar multiple. Further more, an integral plays the role of invariant trace. Namely, for all $x,y\in H$,
$$\lambda(xy)=\lambda(yx),~~~~\lambda(S(x))=\lambda(x).$$
In the following, we choose a normalization so that $\lambda(\Lambda)=\lambda(S(\Lambda))=1$.

A {\it quasitriangular}  Hopf algebra $H$ is equipped with an $R$-matrix $R\in H\otimes H$ such that $$R\Delta(x)=R\Delta^{op}(x),~~(id\otimes\Delta)R=R_{13}R_{12},~~(\Delta\otimes id)R=R_{13}R_{23},$$ where $R_{ij}\in H\otimes H\otimes H$ be obtained
from $R=\sum_k s_k\otimes t_k$ by inserting the unit $1$ into the
tensor factor labeled by the index in $\{1,2,3\}\backslash \{i,j\}$. The {\it Drinfeld element} $u=\sum_k S(t_k)s_k$ satisfies $S^2(x)=uxu^{-1}$ for $x\in H$. Note that if $H$ is semisimple, then $u$ is a central element.

With $R$-matrix, one can define an algebra anti-homomorphism $f_{R_{21}R}:H^*\rightarrow H$, called \textit{Drinfeld map}. For any $p\in H^*$,
\begin{equation*}
f_{R_{21}R}(p):=\sum_{i,j}p(t'_js_i)s'_jt_i\tag{2.1}
\end{equation*}
where $R=\sum_i s_i\otimes t_i=\sum_j s'_j\otimes t'_j$. If $f_{R_{21}R}$ is a linear isomorphism, then $H$ is said to be {\it factorizable}. For a semisimple factorizable Hopf algebra $H$, the Drinfeld map sends an integral to a cointegral. That is,
$f_{R_{21}R}(\lambda)=\Lambda$ (see~\cite{CoWe2}).

A quasitrangular Hopf
algebra is said to be {\it ribbon} if there exists a central element
$\theta$ such that
$$\Delta(\theta)=(R_{21} R)^{-1}(\theta\otimes\theta),\ \ \epsilon(\theta)=1,\ \ \textrm{and} \ \ S(\theta)=\theta.$$
Here $R_{21}=\sum_k t_k\otimes s_k$. It can be shown that the \textit{balancing element} $G=u\theta^{-1}$ induces the antipode square $S^2$. That is,
$S^2(x)=GxG^{-1}$ for $x\in H$. When $H$ is semisimple, it has a canonical ribbon element $\theta=u$ and $G=1$.

The following are examples of finite dimensional semisimple factorizable ribbon Hopf algebras.

\begin{enumerate}

\item The group algebra $\mathbb{C}[\mathbb{Z}]$ becomes quasitriangular equipped with  $$R=\frac{1}{n}\sum^{n}_{a,b=0}e^{-\frac{2\pi iab}{n}}g^a\otimes g^b.$$ 
Then $\mathbb{C}[\mathbb{Z}_n]$ is ribbon since it is semisimple. Note that $\mathbb{C}[\mathbb{Z}_n]$ is factorizable if and only if $n$ is odd.

\item For a finite dimensional semisimple Hopf algebra $H$, its Drinfeld double $D(H)$ is semisimple and factorizable. 
\end{enumerate}

In the following, we review Hennings and Kuperberg invariants in the setting of finite dimensional semisimple factorizable Hopf algebras that are ribbon Hopf algebras with the canonical ribbon elements.

\subsection{Hennings invariant}

In 1990s, Hennings constructed invariant for any closed oriented 3-manifold using ribbon Hopf algebras with certain non-degenerated condition \cite{H}. Then Kauffman and Radford reformulated Hennings' construction via unoriented surgery diagrams \cite{KR}. Given a semisimple ribbon Hopf algebra $H$,  one can associate a regular isotopy invariant $TR(L,H)$ to a framed link $L$ as follows: given any link diagram of $L$, decorate each crossing
with the tensor factors from the
$R$-matrix $R=\sum_i s_i\otimes t_i$ as below.

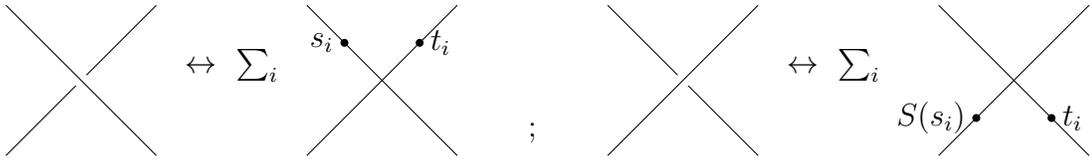
\vspace{.1in}

\begin{tikzpicture}\centering
\draw (1,1)--(-1,-1);
\draw [color=white,line width=2mm] (-1,1)--(1,-1);
\draw (-1,1)--(1,-1);
\node (=) at (2,0.2) {$\leftrightarrow~\sum_i$};
\begin{scope}[xshift = 4cm]
\draw (-1,1)--(1,-1);
\draw (1,1)--(-1,-1);
\fill[black, opacity=1] (-0.5,0.5) circle (1.5pt) node[anchor=east] {$s_i$};
\fill[black, opacity=1] (0.5,0.5) circle (1.5pt) node[anchor=west] {$t_i$};
\end{scope}
\node (=) at (6,-0.7) {$;$};
\begin{scope}[xshift = 8cm]
\draw (-1,1)--(1,-1);
\draw [color=white,line width=2mm] (-1,-1)--(1,1);
\draw (-1,-1)--(1,1);
\node (=) at (2,0.2) {$\leftrightarrow~\sum_i$};
\begin{scope}[xshift = 4.4cm]
\draw (-1,1)--(1,-1);
\draw (1,1)--(-1,-1);
\fill[black, opacity=1] (-0.5,-0.5) circle (1.5pt) node[anchor=east] {$S(s_i)$};
\fill[black, opacity=1] (0.5,-0.5) circle (1.5pt) node[anchor=west] {$t_i$};
\end{scope}
\end{scope}
\end{tikzpicture}

\vspace{.1in}

Once all the crossings of $L$ have been decorated, one gets a labeled diagram immersed in the plane, whose crossings become $4$-valent vertices. The Hopf algebra elements can slide across maxima or minima with the antipode action or its inverse as below.\\
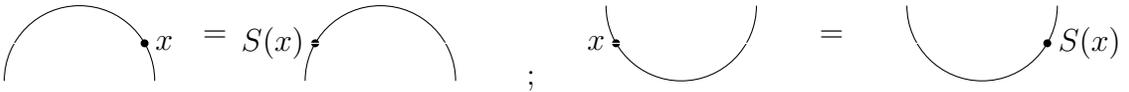
\[
\begin{tikzpicture}
  \draw [name path=circle-1] (-1,0) arc(0:180:1);
  \draw [color=white, name path=line](-3,0.5)--(3,0.5);
  \fill[black, opacity=1, name intersections={of= circle-1 and line}] (intersection-1) circle (1.5pt) node[anchor=west] {$x$};
  \node (=) at (-0.2,0.6) {=};
  \begin{scope}[xshift = 4cm]
  \draw [name path=circle-1] (-1,0) arc(0:180:1);
  \draw [color=white, name path=line](-3,0.5)--(3,0.5);
  \fill[black, opacity=1, name intersections={of= circle-1 and line}] (intersection-2) circle (1.5pt) node[anchor=east] {$S(x)$};
  \end{scope}
  \node (;) at (4,0) {$;$};
  \begin{scope}[xshift = 4cm]
  \draw [name path=circle-1] (1,1) arc(180:360:1);
  \draw [color=white, name path=line](-3,0.5)--(3,0.5);
  \fill[black, opacity=1, name intersections={of= circle-1 and line}] (intersection-1) circle (1.5pt) node[anchor=east] {$x$};
  \node (=) at (4,0.6) {=};
  \begin{scope}[xshift = 4cm]
  \draw [name path=circle-1] (1,1) arc(180:360:1);
  \draw [color=white, name path=line](-3,0.5)--(3,0.5);
  \fill[black, opacity=1, name intersections={of= circle-1 and line}] (intersection-2) circle (1.5pt) node[anchor=west] {$S(x)$};
  \end{scope}
  \end{scope}
\end{tikzpicture}
\]
Slide all the Hopf algebra elements on the same component into one vertical portion and multiply them upwards so that a product $w_i\in H$ is assigned to each component.\\
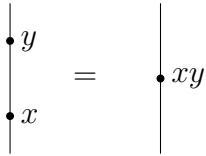
\[
\begin{tikzpicture}
\draw (0,1)--(0,-1);
\fill[black, opacity=1] (0,0.5) circle (1.5pt) node[anchor=west] {$y$};
\fill[black, opacity=1] (0,-0.5) circle (1.5pt) node[anchor=west] {$x$};
\node (=) at (1,0) {=};
\begin{scope}[xshift = 2cm]
\draw (0,1)--(0,-1);
\fill[black, opacity=1] (0,0) circle (1.5pt) node[anchor=west] {$xy$};
\end{scope}
\end{tikzpicture}
\]
Define $$TR(L,H)=\lambda(w_1)\cdots \lambda(w_{c(L)}),$$ where $c(L)$ denotes the number of components of $L$. Note that the full definition of $TR(L,H)$ evolves the Whitney degree of each component and powers of the balancing element $G$. (See \cite{KR}, \cite{ChaWa}). Here $G=1$ in our setting of canonical ribbon structure. 

If
$\lambda(\theta)\lambda(\theta^{-1})\neq 0$, which is always true when $H$ is factorizable \cite{CoWe2}, then
\begin{equation*}
Z_{Henn}(M(L),H)=[\lambda(\theta)\lambda(\theta^{-1})]^{-\frac{c(L)}{2}}[\lambda(\theta)/\lambda(\theta^{-1})]^{-\frac{\sigma(L)}{2}}TR(L,H)\tag{2.2}
\end{equation*}
is an invariant of the closed oriented $3$-manifold $M(L)$ obtained from surgery on the framed link $L$, where $\sigma(L)$ denotes the signature of the framing matrix of $L$.

\subsection{Kuperberg invariant}

From a finite dimensional Hopf algebra $H$, Kuperberg constructed invariant $Z_{Kup}(M,f,H)$ for any closed oriented 3-manifold $M$ with framing $f$ (\cite{Ku}). In the following, we recall Kuperberg invariant in the setting of semisimple Hopf algebra. For full definition, see \cite{Ku} and \cite{ChaWa}. 

Given a closed oriented $3$-manifold $M$, it can be obtained by gluing two handlebodies of genus $g$ along their boundaries. A Heegaard diagram consists of two families of $g$ simple closed curves on a genus $g$ closed oriented surface $F$ which tell us how to glue the two handlebodies. One family of simple closed curve are referred as lower circles and the other family as upper circles. Note that this choice is arbitrary. The orientation of $M$ induces an orientation on its Heegaard surface $F$, by the convention that a positive tangent basis at a point on $F$ extends to a positive basis for $M$ by appending a normal vector that points from the lower side to the upper side. 

A non-vanishing tangent vector field on $M$ is referred as a \textit{combing} on $M$. Any combing on $M$ can be represented completely  by a combing on the Heegaard diagram, which is a vector field on $F$ with $2g$ singularities of index $-1$, one on each circle, and one singularity of index $+2$ disjoint from all circles. The singularity of index $-1$ on a given circle, which is called the \textit{base point} of the circle, should not be a crossing and the two outward-pointing vectors from the base point should be tangent to the circle. Kuperberg showed that any combing on a Heegaard diagram of $M$ can be extended to a combing on $M$; conversely, any combing on $M$ is homotopic to an extension of some combing on the Heegaard diagram.

A \textit{framing} on $M$ consists of three orthogonal non-vanishing vector fields on $M$. It suffices to described a framing by two orthogonal non-vanishing vector field $b_1$ and $b_2$ (the third one is determined by the orientation of $M$). Suppose $b_1$ has been represented on the Heegaard diagram. Then $b_2$ can be described using \textit{twist front} that encodes how $b_2$ rotates relative to $b_1$.  For a factorizable Hopf algebra, $b_2$ contributes powers of the antipode square $S^2$ in the 3-manifold invariant (see \cite{Ku} and \cite{ChaWa} for detail) . When the Hopf algebra is semisimple, $S^2=id$ and so Kuperberg invariant only depends on a combing rather than a framing. 

To define Kuperberg invariant, orient all Heegaard circles
according to the orientation of $M$. Let $b_1$ be a combing on the Heegaard diagram. For each point $p$ on a circle $c$ with base point $o_c$, $\psi(p)$ is defined to be the
counterclockwise rotation of the tangent to $c$ relative to $b_1$ from $o_c$ to $p$ in units of $1=360^\circ$. If $p$ is a crossing of a lower circle and a upper circle, then two rotation angles $\psi_l(p)$ and $\psi_u(p)$ are defined, respectively. Then an integer $$a_p=2(\psi_l(p)-\psi_u(p))-\frac{1}{2}$$ is assigned to the crossing $p$.

The algorithm to write down Kuperberg invariant is as follows: assign a cointegral to each lower circle; do comultiplication for each cointegral and label each crossing with tensor product factors in the direction of lower circles  starting from their based points; apply $S^{a_p}$ to the tensor factor labeled at crossing $p$; multiply all labels in the direction of upper circles starting from their based points; evaluate the resulting products using integrals.
In short, the invariant is a summation:
\begin{eqnarray*}
  Z_{Kup}(M,f,H)=\sum\limits_{(\Lambda)}\prod_{\substack{upper \\ circles}}\lambda\big(\cdots S^{a_i}(\Lambda_{(i)})\cdots \big)
\end{eqnarray*}
In next section, an example of Kuperberg invariant is given for the framed Heegaard diagram in Fig.~\ref{fig:genus two}.

\section{Proof of Theorem}

In this section, we prove the theorem. To compute the Kuperberg invariant for a closed 3-manifold $M$, we construct a suitable framing $f$ on its Heegaard diagram. On the other hand, we calculate the Hennings invariant for $M\#\overline{M}$ whose surgery diagram is given by the chain-mail link (\cite{Ro}), where $\overline{M}$ denotes the manifold with the opposite orientation as $M$.

\subsection{$Z_{Kup}(M,f,H)$}

In the following, a 2-sphere $S^2$ is regarded as a plane together with the point at infinity. A genus $g$ Heegaard surface is obtained by attaching $g$ 1-handles (not drawn) to $2g$ discs in the plane. To do Heegaard decomposition, we can glue one genus $g$ handlebody from below and another one from above. By isotope, the attaching circles of the lower handlebody (called lower circles) can be always chosen to be the meridians of the handles. In our figures, they are drawn as horizontal lines from the left  disk to the right disk. Parts of them go through the handles above the plane and so are not drawn. For instance, Fig.~\ref{fig:Heegaard L(5,2)} and Fig.~\ref{fig:genus two} are the Heegaard diagrams of Lens space $L(5,2)$ and Poincare homology 3-sphere, respectively.

\begin{figure}\centering
\begin{tikzpicture}
  \draw [color=white, name path=hline-1] (-1.5,0.2)--(-2,0.2);
  \draw [color=white, name path=hline-2] (-1.5,-0.2)--(-2,-0.2);
  \draw [color=white, name path=hline-3] (2.5,0.2)--(3,0.2);
  \draw [color=white, name path=hline-4] (2.5,-0.2)--(3,-0.2);
  \draw [thick,name path=circle-1](-1.5,0) circle (0.5);
  \draw [thick,name path=circle-2](2.5,0) circle (0.5);
  \draw [thick, green](2,0)--(-1,0);
  \draw [thick, blue](-0.5,-0.4)--(-0.5,0.4);
  \draw [thick, blue](0,-0.4)--(0,0.4);
  \draw [thick, blue](0.5,-0.4)--(0.5,0.4);
  \draw [thick, blue](1,-0.4)--(1,0.4);
  \draw [thick, blue](1.5,-0.4)--(1.5,0.4);
  \draw [thick, blue](-0.5,0.4) arc(0:180:1 and 0.6);
  \draw [thick, blue](-2.5,0.4) arc(180:270:0.3 and 0.2);
  \draw [thick, blue, name intersections={of=circle-1 and hline-1}] (-2.2,0.2)--(intersection-1);
  \draw [thick, blue](0,0.4) arc(0:180:1.5 and 1.1);
  \draw [thick, blue](-3,0.4) arc(180:270:0.8 and 0.6);
  \draw [thick, blue, name intersections={of=circle-1 and hline-2}] (-2.2,-0.2)--(intersection-1);
  \draw [thick, blue](1,-0.4) arc(180:360:1.5 and 1.1);
  \draw [thick, blue](4,-0.4) arc(0:90:0.8 and 0.6);
  \draw [thick, blue, name intersections={of=circle-2 and hline-3}] (3.2,0.2)--(intersection-1);
  \draw [thick, blue](1.5,-0.4) arc(180:360:1 and 0.6);
  \draw [thick, blue](3.5,-0.4) arc(0:90:0.3 and 0.2);
  \draw [thick, blue, name intersections={of=circle-2 and hline-4}] (3.2,-0.2)--(intersection-1);
  \draw [thick, blue](1.5,0.4) arc(180:90:1 and 0.6);
  \draw [thick, blue](2.5,1) arc(90:0:2 and 1.4);
  \draw [thick, blue](4.5,-0.4) arc(360:180:2 and 1.5);
  \draw [thick, blue](1,0.4) arc(180:90:1.5 and 1.1);
  \draw [thick, blue](2.5,1.5) arc(90:0:2.5 and 1.9);
  \draw [thick, blue](5,-0.4) arc(360:180:2.5 and 2);
  \draw [thick, blue](0.5,0.4) arc(180:90:2 and 1.6);
  \draw [thick, blue](2.5,2) arc(90:0:3 and 2.4);
  \draw [thick, blue](5.5,-0.4) arc(360:180:3 and 2.5);
\end{tikzpicture}
\caption{}
\label{fig:Heegaard L(5,2)}
\end{figure}

The attaching circles of the upper handle body are called upper circles. There is a pairing $\sigma$ between the set of upper circles and the set of handles. Namely, a circle $c$ can be matched with exactly one handle $\sigma(c)$ such that an arc of $c$ passes between the two attaching discs of the handle $\sigma(c)$ and there are no other upper circles between this arc and the right attaching disc of $\sigma(c)$. This pairing can be obtained by isotopy.  Fig.~\ref{fig:isotopy} shows the isotopy to obtain such pairing. First, if there is no any arc passing between the two attaching discs of some handle, then an arc can be pulled to pair with this handle (e.g., see $c_2$ and $\sigma(c_2)$ in Fig.~\ref{fig:isotopy}). If an arc is in the middle of $c$ and $\sigma(c)$, then this arc can be pushed to the left attaching disc through the handle above the page (e.g., see $c_1$ and $\sigma(c_1)$ in Fig.~\ref{fig:isotopy}). 

\begin{figure}\centering
\begin{tikzpicture}[scale=0.6]
\begin{scope}
\draw [thick, name path=circle-1] (0,0) circle(1);
\draw [thick, name path=circle-2] (6,0) circle(1);
\draw [thick, name path=circle-3] (0,5) circle(1);
\draw [thick, name path=circle-4] (6,5) circle(1);
\draw[red, line width=0.33mm] (-1.3,3) to [out=0,in=-100] (3,7);
\draw[blue, line width=0.33mm] (-1.3,2) to [out=0,in=-100] (5,7);
\draw (2.35,5) node[anchor=west] {$c_1$};
\draw (2.8,2.9) node[anchor=west] {$c_2$};
\end{scope}
\begin{scope}[xshift=10cm]
\draw (0,3) node[anchor=south] {Isotopy};
\draw [->,line width=0.5mm] (-1.5,3)--(1.5,3);
\end{scope}
\begin{scope}[xshift=14cm]
\draw [thick, name path=circle-1] (0,0) circle(1);
\draw [thick, name path=circle-2] (6,0) circle(1);
\draw [thick, name path=circle-3] (0,5) circle(1);
\draw [thick, name path=circle-4] (6,5) circle(1);
\draw[red, line width=0.33mm] (-1.3,3) to [out=0,in=-90] (4.3,5) to [out=90,in=-90] (4,7);
\draw[blue, line width=0.33mm] (-1.3,2) to [out=25,in=-100] (4.3,-1) to [out=80,in=-140] (6,5) to [out=160,in=-80] (5,7);
\fill [white] (6,5) circle(0.96);
\draw[blue, line width=0.33mm] (0,5.6) to [out=0,in=90] (1.7,5) to [out=-90,in=0] (0,4.4);
\fill [white] (0,5) circle(0.96);
\draw (4.3,5) node[anchor=east] {$c_1$};
\draw (4.3,0) node[anchor=east] {$c_2$};
\draw (6,5) node {$\sigma(c_1)$};
\draw (6,0) node {$\sigma(c_2)$};
\end{scope}
\end{tikzpicture}
\caption{}
\label{fig:isotopy}
\end{figure}

Let us fix a choice of the pairing $\sigma$ and set up a combing on the Heegaard surface, which is described in Fig.~\ref{fig:vector field}. First, all lower circles are represented by the horizontal lines. The index $-1$ singularity (base point) on any lower circle $c_k^L$ is placed next to the left attaching disc so that no upper circles pass between them. For a upper circle $c_k^U$, by the pairing $\sigma$, there are no other upper circles passing between $c_k^U$ and $\sigma(c_k^U)$. Then the index $-1$ singularity (base point) on $c_k^U$ is placed at the point closest to the right attaching disc of $\sigma(c_k^U)$. When $c_k^L$ is close to the singularity on $c_k^U$, it goes slightly off to avoid this singularity. 

\begin{figure}\centering
\begin{tikzpicture}[scale=0.8, fill=white]
  \draw [thick, name path=circle-1] (0,0) circle(1);
  \draw [thick, name path=circle-2] (13,0) circle(1);
  \draw [dashed] (-2,-2) to [out=0,in=-30] (2,-0.8) to [out=150,in=-60] (0,0); 
  \draw [->,line width=0.4mm] (0.15,-1.85)--(0.2,-1.84);
  \draw [->,line width=0.4mm] (1.4,-0.59)--(1.35,-0.58);
  \draw [dashed] (-2,2) to [out=0,in=30] (2,0.8) to [out=-150,in=60] (0,0); 
  \draw [->,line width=0.4mm] (0.15,1.86)--(0.2,1.84);
  \draw [->,line width=0.4mm] (1.4,0.59)--(1.35,0.58);
  \draw [dashed] (-2,3) to [out=0,in=90] (4,0) to [out=-90,in=0] (-2,-3); 
  \draw [->,line width=0.4mm] (0.1,2.87)--(0.15,2.86);
  \draw [->,line width=0.4mm] (3.4,1.65)--(3.45,1.6);
  \draw [->,line width=0.4mm] (0.1,-2.87)--(0.15,-2.86);
  \draw [->,line width=0.4mm] (3.4,-1.65)--(3.45,-1.6); 
  \draw [dashed] (15,-2) to [out=180,in=-150] (11,-0.8) to [out=30,in=-120] (13,0); 
  \draw [->,line width=0.4mm] (12.8,-1.84)--(12.85,-1.85);
  \draw [->,line width=0.4mm] (11.57,-0.6)--(11.52,-0.61);
  \draw [dashed] (15,2) to [out=180,in=150] (11,0.8) to [out=-30,in=120] (13,0); 
  \draw [->,line width=0.4mm] (12.8,1.84)--(12.85,1.85);
  \draw [->,line width=0.4mm] (11.57,0.6)--(11.52,0.61);
  \draw [dashed] (15,3) to [out=180,in=90] (9,0) to [out=-90,in=180] (15,-3);  
  \draw [->,line width=0.4mm] (12.85,2.89)--(12.9,2.9);
  \draw [->,line width=0.4mm] (9.55,1.55)--(9.6,1.6);
  \draw [->,line width=0.4mm] (12.85,-2.89)--(12.9,-2.9);
  \draw [->,line width=0.4mm] (9.55,-1.55)--(9.6,-1.6); 
  \draw [dashed] (-2,4) to [out=0,in=180] (6.5,2.7) to [out=0,in=180] (15,4); 
  \draw [->,line width=0.4mm] (0.1,3.82)--(0.15,3.81);
  \draw [->,line width=0.4mm] (6.5,2.7)--(6.55,2.7);
  \draw [->,line width=0.4mm] (12.75,3.82)--(12.8,3.83);
  \draw [dashed, line width=0.5mm] (-2,0)--(15,0);
  \draw [->, dashed, line width=0.4mm] (-1.44,0)--(-1.43,0);
  \draw [->, dashed, line width=0.4mm] (2.7,0)--(2.65,0);
  \draw [->, dashed, line width=0.4mm] (5.25,0)--(5.3,0);
  \draw [->, dashed, line width=0.4mm] (7.86,0)--(7.87,0);
  \draw [->, dashed, line width=0.4mm] (10.16,0)--(10.15,0);
  \draw [->, dashed, line width=0.4mm] (14.52,0)--(14.53,0);
  \draw [dashed] (-2,-4) to [out=0,in=180] (6.5,-2.7) to [out=0,in=180] (15,-4); 
  \draw [->,line width=0.4mm] (0.1,-3.81)--(0.15,-3.8);
  \draw [->,line width=0.4mm] (6.5,-2.7)--(6.55,-2.7);
  \draw [->,line width=0.4mm] (12.75,-3.82)--(12.8,-3.83);
  \fill [white] (0,0) circle(1);
  \fill [white] (13,0) circle(1);
  \fill [white] (0,-5) circle(1);
  \fill [white] (13,-5) circle(1);
  \draw [green, thick, rounded corners=5pt] (1,0)--(8.4,0)--(8.7,-0.5)--(9.3,-0.5)--(9.6,0)--(12,0);
  \draw [blue, thick] (9,-1.5)--(9,1.5);
  \draw (4.5,0) node[anchor=north] {$c_k^L$};
  \draw (9,1) node[anchor=east] {$c_k^U$};
  \draw (13,0) node {$\sigma(c_k^U)$};
\end{tikzpicture}
\caption{}
\label{fig:vector field}
\end{figure}

Fig.~\ref{fig:vector field} provides a local picture around one handle. The dashed lines indicate the flow of the vector field. The vector field flows parallel through the handle. Several copies of such local vector fields stack up and down to form a global vector field. An genus two example is drawn in Fig.~\ref{fig:genus two}. Finally, the index  $+2$ singularity is located at the infinity. 

The orientation of the Heegaard surface is chosen by setting its normal vector upwards through the paper. The lower circles are oriented from left to right while the upper circles are going upwards from the base point. To write down the Kuperberg invariant from the diagram, one splits cointegrals into coproduct factors and label them at the intersections of the lower and upper circles then multiply the coproduct factors following the direction of upper circles. $S^{a_p}$ acts on the coproduct factors according to the angle relative to the combing. Now we calculate the change of the power of $S$ when traveling along the circles.

\begin{lemma}\label{Kuperberglemma-1}
(1) The power of $S$ remains unchanged when traveling along a vertical arc.\\
(2) The power of $S$ increases by $1$ when passing through an extremum in a counterclockwise direction;\\
(3) The power of $S$ decreases by $1$ when passing through an extremum in a clockwise direction;
\end{lemma}

\begin{proof}
(1) is obvious because the angle relative to the combing does not change when traveling transversely along a vertical arc. Because of the same reason, y. If the arc passes through two handles, as Fig.~\ref{fig:S change-2}, then the calculation can be reduced to the case within one handle as Fig.~\ref{fig:S change-1}. Thus (2) be examined just for the local picture one handle as Fig.~\ref{fig:S change-1}. 

Suppose $S^{a_j}(\Lambda_{j})$ and  $S^{a_{j+1}}(\Lambda'_{j+1})$ are successive terms in the product along a upper circle. Here $\Lambda$ and $\Lambda'$ are two copies of cointegral. Then
$$a_{j+1}-a_j=2(\psi_L(\Lambda'_{j+1})-\psi_L(\Lambda_{j}))-2(\psi_U(\Lambda'_{j+1})-\psi_U(\Lambda_{j}))$$
In the case shown in Fig.~\ref{fig:S change-1}, we move from the coproduct factor $\Lambda_{j}$ to $\Lambda_{j+1}$ through an extremum counterclockwise, then 
$$a_{j+1}-a_j=2\cdot0-2(-\frac{1}{2})=1$$
The clockwise case (3) can be verified similarly. 

\begin{figure}\centering
\begin{tikzpicture}[scale=0.8]\footnotesize
  \draw [thick, name path=circle-1] (0,0) circle(1);
  \draw [thick, name path=circle-2] (13,0) circle(1);
  \draw [dashed] (-2,-2) to [out=0,in=-30] (2,-0.8) to [out=150,in=-60] (0,0); 
  \draw [->,line width=0.4mm] (0.15,-1.85)--(0.2,-1.84);
  \draw [->,line width=0.4mm] (1.4,-0.59)--(1.35,-0.58);
  \draw [dashed] (-2,2) to [out=0,in=30] (2,0.8) to [out=-150,in=60] (0,0); 
  \draw [->,line width=0.4mm] (0.15,1.86)--(0.2,1.84);
  \draw [->,line width=0.4mm] (1.4,0.59)--(1.35,0.58);
  \draw [dashed] (-2,3) to [out=0,in=90] (4,0) to [out=-90,in=0] (-2,-3); 
  \draw [->,line width=0.4mm] (0.1,2.87)--(0.15,2.86);
  \draw [->,line width=0.4mm] (3.4,1.65)--(3.45,1.6);
  \draw [->,line width=0.4mm] (0.1,-2.87)--(0.15,-2.86);
  \draw [->,line width=0.4mm] (3.4,-1.65)--(3.45,-1.6); 
  \draw [dashed] (15,-2) to [out=180,in=-150] (11,-0.8) to [out=30,in=-120] (13,0); 
  \draw [->,line width=0.4mm] (12.8,-1.84)--(12.85,-1.85);
  \draw [->,line width=0.4mm] (11.57,-0.6)--(11.52,-0.61);
  \draw [dashed] (15,2) to [out=180,in=150] (11,0.8) to [out=-30,in=120] (13,0); 
  \draw [->,line width=0.4mm] (12.8,1.84)--(12.85,1.85);
  \draw [->,line width=0.4mm] (11.57,0.6)--(11.52,0.61);
  \draw [dashed] (15,3) to [out=180,in=90] (9,0) to [out=-90,in=180] (15,-3);  
  \draw [->,line width=0.4mm] (12.85,2.89)--(12.9,2.9);
  \draw [->,line width=0.4mm] (9.55,1.55)--(9.6,1.6);
  \draw [->,line width=0.4mm] (12.85,-2.89)--(12.9,-2.9);
  \draw [->,line width=0.4mm] (9.55,-1.55)--(9.6,-1.6); 
  \draw [dashed] (-2,5)--(15,5);
  \draw [->,line width=0.4mm] (1.6,5)--(1.7,5);
  \draw [->,line width=0.4mm] (6.5,5)--(6.6,5);
  \draw [->,line width=0.4mm] (11.4,5)--(11.5,5);
  \draw [dashed] (-2,4) to [out=0,in=180] (6.5,2.7) to [out=0,in=180] (15,4); 
  \draw [->,line width=0.4mm] (0.1,3.82)--(0.15,3.81);
  \draw [->,line width=0.4mm] (6.5,2.7)--(6.55,2.7);
  \draw [->,line width=0.4mm] (12.75,3.82)--(12.8,3.83);
  \draw [dashed] (-2,0)--(15,0);
  \draw [->, dashed, line width=0.4mm] (-1.44,0)--(-1.43,0);
  \draw [->, dashed, line width=0.4mm] (2.7,0)--(2.65,0);
  \draw [->, dashed, line width=0.4mm] (4.95,0)--(5,0);
  \draw [->, dashed, line width=0.4mm] (7.76,0)--(7.77,0);
  \draw [->, dashed, line width=0.4mm] (10.16,0)--(10.15,0);
  \draw [->, dashed, line width=0.4mm] (14.52,0)--(14.53,0);
  \draw [dashed] (-2,-4) to [out=0,in=180] (6.5,-2.7) to [out=0,in=180] (15,-4); 
  \draw [->,line width=0.4mm] (0.1,-3.81)--(0.15,-3.8);
  \draw [->,line width=0.4mm] (6.5,-2.7)--(6.55,-2.7);
  \draw [->,line width=0.4mm] (12.75,-3.82)--(12.8,-3.83);
  \draw [dashed] (-2,-5)--(15,-5);
  \draw [->,line width=0.4mm] (1.6,-5)--(1.7,-5);
  \draw [->,line width=0.4mm] (6.5,-5)--(6.6,-5);
  \draw [->,line width=0.4mm] (11.4,-5)--(11.5,-5);
  \fill [white] (0,0) circle(1);
  \fill [white] (13,0) circle(1);
  \fill [white] (0,-6) circle(1);
  \fill [white] (13,-6) circle(1);
  \draw [green, thick, rounded corners=5pt] (1,0)--(8.4,0)--(8.7,-0.5)--(9.3,-0.5)--(9.6,0)--(12,0);
  \draw [blue, thick] (4.3,-1.5) to [out=90,in=180] (6,1.5) to [out=0,in=90] (7.5,-1.5);
  \draw (5.8,0) node[anchor=north] {$S^{a_j+1}(\Lambda'_{j+1})$};
  \draw (8.1,0.1) node[anchor=south] {$S^{a_j}(\Lambda_j)$};
\end{tikzpicture}  
\caption{}
\label{fig:S change-1}
\end{figure}

\begin{figure}\centering
\begin{tikzpicture}[scale=0.8]\footnotesize
\begin{scope}
  \draw [thick, name path=circle-1] (0,0) circle(1);
  \draw [thick, name path=circle-2] (13,0) circle(1);
  \draw [dashed] (-2,-2) to [out=0,in=-30] (2,-0.8) to [out=150,in=-60] (0,0); 
  \draw [->,line width=0.4mm] (0.15,-1.85)--(0.2,-1.84);
  \draw [->,line width=0.4mm] (1.4,-0.59)--(1.35,-0.58);
  \draw [dashed] (-2,2) to [out=0,in=30] (2,0.8) to [out=-150,in=60] (0,0); 
  \draw [->,line width=0.4mm] (0.15,1.86)--(0.2,1.84);
  \draw [->,line width=0.4mm] (1.4,0.59)--(1.35,0.58);
  \draw [dashed] (-2,3) to [out=0,in=90] (4,0) to [out=-90,in=0] (-2,-3); 
  \draw [->,line width=0.4mm] (0.1,2.87)--(0.15,2.86);
  \draw [->,line width=0.4mm] (3.4,1.65)--(3.45,1.6);
  \draw [->,line width=0.4mm] (0.1,-2.87)--(0.15,-2.86);
  \draw [->,line width=0.4mm] (3.4,-1.65)--(3.45,-1.6); 
  \draw [dashed] (15,-2) to [out=180,in=-150] (11,-0.8) to [out=30,in=-120] (13,0); 
  \draw [->,line width=0.4mm] (12.8,-1.84)--(12.85,-1.85);
  \draw [->,line width=0.4mm] (11.57,-0.6)--(11.52,-0.61);
  \draw [dashed] (15,2) to [out=180,in=150] (11,0.8) to [out=-30,in=120] (13,0); 
  \draw [->,line width=0.4mm] (12.8,1.84)--(12.85,1.85);
  \draw [->,line width=0.4mm] (11.57,0.6)--(11.52,0.61);
  \draw [dashed] (15,3) to [out=180,in=90] (9,0) to [out=-90,in=180] (15,-3);  
  \draw [->,line width=0.4mm] (12.85,2.89)--(12.9,2.9);
  \draw [->,line width=0.4mm] (9.55,1.55)--(9.6,1.6);
  \draw [->,line width=0.4mm] (12.85,-2.89)--(12.9,-2.9);
  \draw [->,line width=0.4mm] (9.55,-1.55)--(9.6,-1.6); 
  \draw [dashed] (-2,4) to [out=0,in=180] (6.5,2.7) to [out=0,in=180] (15,4); 
  \draw [->,line width=0.4mm] (0.1,3.82)--(0.15,3.81);
  \draw [->,line width=0.4mm] (6.5,2.7)--(6.55,2.7);
  \draw [->,line width=0.4mm] (12.75,3.82)--(12.8,3.83);
  \draw [dashed] (-2,0)--(15,0);
  \draw [->, dashed, line width=0.4mm] (-1.44,0)--(-1.43,0);
  \draw [->, dashed, line width=0.4mm] (2.7,0)--(2.65,0);
  \draw [->, dashed, line width=0.4mm] (4.95,0)--(5,0);
  \draw [->, dashed, line width=0.4mm] (7.36,0)--(7.37,0);
  \draw [->, dashed, line width=0.4mm] (10.16,0)--(10.15,0);
  \draw [->, dashed, line width=0.4mm] (14.52,0)--(14.53,0);
  \draw [dashed] (-2,-4) to [out=0,in=180] (6.5,-2.7) to [out=0,in=180] (15,-4); 
  \draw [->,line width=0.4mm] (0.1,-3.81)--(0.15,-3.8);
  \draw [->,line width=0.4mm] (6.5,-2.7)--(6.55,-2.7);
  \draw [->,line width=0.4mm] (12.75,-3.82)--(12.8,-3.83);
  \fill [white] (0,0) circle(1);
  \fill [white] (13,0) circle(1);
  \fill [white] (0,-10) circle(1);
  \fill [white] (13,-10) circle(1);
  \draw [green, thick, rounded corners=5pt] (1,0)--(8.4,0)--(8.7,-0.5)--(9.3,-0.5)--(9.6,0)--(12,0);
  \draw [blue, thick] (5.2,1.5) to [out=-90,in=70] (4,-5);
  \draw (6.7,0) node[anchor=south] {$S^{a_j-1}(\Lambda'_{j+1})$};
\end{scope}
\begin{scope}[yshift=-10cm]
  \draw [thick, name path=circle-1] (0,0) circle(1);
  \draw [thick, name path=circle-2] (13,0) circle(1);
  \draw [dashed] (-2,-2) to [out=0,in=-30] (2,-0.8) to [out=150,in=-60] (0,0); 
  \draw [->,line width=0.4mm] (0.15,-1.85)--(0.2,-1.84);
  \draw [->,line width=0.4mm] (1.4,-0.59)--(1.35,-0.58);
  \draw [dashed] (-2,2) to [out=0,in=30] (2,0.8) to [out=-150,in=60] (0,0); 
  \draw [->,line width=0.4mm] (0.15,1.86)--(0.2,1.84);
  \draw [->,line width=0.4mm] (1.4,0.59)--(1.35,0.58);
  \draw [dashed] (-2,3) to [out=0,in=90] (4,0) to [out=-90,in=0] (-2,-3); 
  \draw [->,line width=0.4mm] (0.1,2.87)--(0.15,2.86);
  \draw [->,line width=0.4mm] (3.4,1.65)--(3.45,1.6);
  \draw [->,line width=0.4mm] (0.1,-2.87)--(0.15,-2.86);
  \draw [->,line width=0.4mm] (3.4,-1.65)--(3.45,-1.6); 
  \draw [dashed] (15,-2) to [out=180,in=-150] (11,-0.8) to [out=30,in=-120] (13,0); 
  \draw [->,line width=0.4mm] (12.8,-1.84)--(12.85,-1.85);
  \draw [->,line width=0.4mm] (11.57,-0.6)--(11.52,-0.61);
  \draw [dashed] (15,2) to [out=180,in=150] (11,0.8) to [out=-30,in=120] (13,0); 
  \draw [->,line width=0.4mm] (12.8,1.84)--(12.85,1.85);
  \draw [->,line width=0.4mm] (11.57,0.6)--(11.52,0.61);
  \draw [dashed] (15,3) to [out=180,in=90] (9,0) to [out=-90,in=180] (15,-3);  
  \draw [->,line width=0.4mm] (12.85,2.89)--(12.9,2.9);
  \draw [->,line width=0.4mm] (9.55,1.55)--(9.6,1.6);
  \draw [->,line width=0.4mm] (12.85,-2.89)--(12.9,-2.9);
  \draw [->,line width=0.4mm] (9.55,-1.55)--(9.6,-1.6); 
  \draw [dashed] (-2,5)--(15,5);
  \draw [->,line width=0.4mm] (1.6,5)--(1.7,5);
  \draw [->,line width=0.4mm] (6.5,5)--(6.6,5);
  \draw [->,line width=0.4mm] (11.4,5)--(11.5,5);
  \draw [dashed] (-2,4) to [out=0,in=180] (6.5,2.7) to [out=0,in=180] (15,4); 
  \draw [->,line width=0.4mm] (0.1,3.82)--(0.15,3.81);
  \draw [->,line width=0.4mm] (6.5,2.7)--(6.55,2.7);
  \draw [->,line width=0.4mm] (12.75,3.82)--(12.8,3.83);
  \draw [dashed] (-2,0)--(15,0);
  \draw [->, dashed, line width=0.4mm] (-1.44,0)--(-1.43,0);
  \draw [->, dashed, line width=0.4mm] (2.7,0)--(2.65,0);
  \draw [->, dashed, line width=0.4mm] (5,0)--(5.05,0);
  \draw [->, dashed, line width=0.4mm] (7.16,0)--(7.17,0);
  \draw [->, dashed, line width=0.4mm] (10.16,0)--(10.15,0);
  \draw [->, dashed, line width=0.4mm] (14.52,0)--(14.53,0);
  \draw [dashed] (-2,-4) to [out=0,in=180] (6.5,-2.7) to [out=0,in=180] (15,-4); 
  \draw [->,line width=0.4mm] (0.1,-3.81)--(0.15,-3.8);
  \draw [->,line width=0.4mm] (6.5,-2.7)--(6.55,-2.7);
  \draw [->,line width=0.4mm] (12.75,-3.82)--(12.8,-3.83);
  \fill [white] (0,0) circle(1);
  \fill [white] (13,0) circle(1);
  \fill [white] (0,-5) circle(1);
  \fill [white] (13,-5) circle(1);
  \draw [green, thick, rounded corners=5pt] (1,0)--(8.4,0)--(8.7,-0.5)--(9.3,-0.5)--(9.6,0)--(12,0);
  \draw [blue, thick] (5.7,1.5) to [out=-90,in=0] (1.5,-3.5) to [out=180,in=-90] (-1.7,0) to [out=90,in=-110] (4,5);
  \draw (6.4,-0.1) node[anchor=north] {$S^{a_j}(\Lambda_j)$};
\end{scope}
\end{tikzpicture}  
\caption{}
\label{fig:S change-2}
\end{figure}

\end{proof}

\begin{example}
Fig.~\ref{fig:genus two} is a combed Heegaard diagram of Poincare homology 3-sphere. Let $\Lambda$ and $\Lambda'$ be two copies of cointegrals. Their coproduct factors are labeled and multiplied going up from the base points of each upper circle respectively. Then
\begin{align*}Z_{Kup}=\sum\limits_{(\Lambda)}&\lambda\left(S(\Lambda'_{(3)})\Lambda_{(1)}\Lambda'_{(5)}S^{-1}(\Lambda_{(3)})\Lambda'_{(1)}S^{-1}(\Lambda_{(5)})\right)\\
&\cdot\lambda\left(\Lambda_{(4)}S(\Lambda'_{(4)})S(\Lambda'_{(2)})\Lambda_{(2)}\Lambda'_{(6)}\Lambda'_{(7)}\Lambda'_{(8)}\right)
\end{align*}
\end{example}

\subsection{$Z_{Henn}(M\#\overline{M},H)$}
Now, we use the chain-mail link to evaluate the Hennings invariant for $M\#\overline{M}$. A Heegaard diagram of $M$ can be turned into a surgery diagram of $M\#\overline{M}$ by pushing the upper circles into the lower handle body slightly. Then the upper circles and the lower circles form a link $L_M$. All these curves are framed by thickening them into thin bands parallel to the Heegaard surface. The resulting link $L_M$ is a surgery presentation for $M\#\overline{M}$ and called a chain-mail link of $M$ (\cite{Ro}). Fig.~\ref{fig:chainmail L(5,2)} is the chain-mail link of the Lens space $L(5,2)$.

\begin{figure}\centering
\begin{tikzpicture}
  \draw (-0.5,-0.4)--(-0.5,0.4);
  \draw (0,-0.4)--(0,0.4);
  \draw (0.5,-0.4)--(0.5,0.4);
  \draw (1,-0.4)--(1,0.4);
  \draw (1.5,-0.4)--(1.5,0.4);
  \draw (-0.5,0.4) arc(0:180:1 and 0.6);
  \draw (-2.5,0.4) arc(180:270:0.3 and 0.2);
  \draw (0,0.4) arc(0:180:1.5 and 1.1);
  \draw (-3,0.4) arc(180:270:0.8 and 0.6);
  \draw (1,-0.4) arc(180:360:1.5 and 1.1);
  \draw (4,-0.4) arc(0:90:0.8 and 0.6);
  \draw (1.5,-0.4) arc(180:360:1 and 0.6);
  \draw (3.5,-0.4) arc(0:90:0.3 and 0.2);
  \draw (1.5,0.4) arc(180:90:1 and 0.6);
  \draw (2.5,1) arc(90:0:2 and 1.4);
  \draw (4.5,-0.4) arc(360:180:2 and 1.5);
  \draw (1,0.4) arc(180:90:1.5 and 1.1);
  \draw (2.5,1.5) arc(90:0:2.5 and 1.9);
  \draw (5,-0.4) arc(360:180:2.5 and 2);
  \draw (0.5,0.4) arc(180:90:2 and 1.6);
  \draw (2.5,2) arc(90:0:3 and 2.4);
  \draw (5.5,-0.4) arc(360:180:3 and 2.5);
  \draw [color=white,line width=2mm] (2.3,0) arc(0:180:1.8 and 0.8);
  \draw (2.3,0) arc(0:180:1.8 and 0.8);
  \draw (-2.2,0.2) arc(270:360:0.2);
  \draw (3.2,0.2) arc(270:180:0.2);
  \draw [color=white,line width=2mm] (3,0.4) arc(0:180:2.5 and 1);
  \draw (3,0.4) arc(0:180:2.5 and 1);
  \draw (-2.2,-0.2) arc(270:360:0.2);
  \draw (3.2,-0.2) arc(270:180:0.2);
  \draw [color=white,line width=2mm] (3,0) arc(0:180:2.5 and 1.1);
  \draw (3,0) arc(0:180:2.5 and 1.1);
  \draw (2.3,0) arc(360:270:0.3);
  \draw (-1.3,0) arc(180:270:0.3);
  \draw (-1,-0.3)--(2,-0.3);
  \draw [color=white,line width=2mm](-0.5,-0.4)--(-0.5,0);
  \draw [color=white,line width=2mm](0,-0.4)--(0,0);
  \draw [color=white,line width=2mm](0.5,-0.4)--(0.5,0);
  \draw [color=white,line width=2mm](1,-0.4)--(1,0);
  \draw [color=white,line width=2mm](1.5,-0.4)--(1.5,0);
  \draw (-0.5,-0.4)--(-0.5,0);
  \draw (0,-0.4)--(0,0);
  \draw (0.5,-0.4)--(0.5,0);
  \draw (1,-0.4)--(1,0);
  \draw (1.5,-0.4)--(1.5,0);
\end{tikzpicture}
\caption{}
\label{fig:chainmail L(5,2)}
\end{figure}

Note that, the signature $\sigma(L_M)$ of the framing matrix of the chain-mail link is zero and $\lambda(\theta)\lambda(\theta^{-1})=1$ for
a factorizable ribbon Hopf algebra (see~\cite{CoWe2}), so the normalization factor in $(2.2)$ 
$$[\lambda(\theta)\lambda(\theta^{-1})]^{-\frac{c(L_M)}{2}}[\lambda(\theta)/\lambda(\theta^{-1})]^{-\frac{\sigma(L_M)}{2}}=1.$$
Thus it is sufficient to find the link invariant $TR(L_M,H)$. Following the algorithm, the chain-mail link is decorated with $R$-matrix that is transformed into the cointegral via the Drinfeld map $(2.1)$. As a result, Lemma \ref{Henningslemma-1} shows that the contribution of the lower circles is equivalent to decorating the upper circles with coproduct factors of cointegrals. Thus we can work on the cointegral decorated diagram to evaluate the Hennings invariant. 
Lemma \ref{Henningslemma-2} states that the self crossings of the upper circles can be resolved and absorbed by cointegrals. The proof of these lemmas are the same as that for Lens spaces (see \cite{ChaWa}).
\begin{lemma}\label{Henningslemma-1}
\[
\begin{tikzpicture}[scale=0.8]
\draw [name path=ellipse] (0,0) ellipse (3cm and 1cm);
\draw (1,-1.5)--(1,1.5);
\draw (2,-1.5)--(2,1.5);
\draw (-1,-1.5)--(-1,1.5);
\draw (-2,-1.5)--(-2,1.5);
\draw (3,0) [color=white,line width=2mm] arc (0:180:3cm and 1cm);
\draw (3,0) arc (0:180:3cm and 1cm);
\draw [color=white,line width=2mm](1,-1.5)--(1,0);
\draw [color=white,line width=2mm](2,-1.5)--(2,0);
\draw [color=white,line width=2mm](-1,-1.5)--(-1,0);
\draw [color=white,line width=2mm](-2,-1.5)--(-2,0);
\draw (1,-1.5)--(1,0);
\draw (2,-1.5)--(2,0);
\draw (-1,-1.5)--(-1,0);
\draw (-2,-1.5)--(-2,0);
\draw [color=white,name path=vline](0.5,-1.5)--(0.5,1.5);
\draw [color=white, line width=2mm, name intersections={of=ellipse and vline}] (intersection-1) arc (80:100:3cm and 1cm);
\draw [dashed, name intersections={of=ellipse and vline}] (intersection-1) arc (80:100:3cm and 1cm);
\draw [color=white, line width=2mm, name intersections={of=ellipse and vline}] (intersection-2) arc (280:260:3cm and 1cm);
\draw [dashed, name intersections={of=ellipse and vline}] (intersection-2) arc (280:260:3cm and 1cm);
\draw (-3,0) circle (0.1pt) node[anchor=east] {$c^L_k$};
\node (=) at (4.2,0.2) {$=$};
\begin{scope}[xshift = 9cm]
\draw (-100pt,30pt)--(-100pt,-30pt);
\draw (-50pt,30pt)--(-50pt,-30pt);
\draw (100pt,30pt)--(100pt,-30pt);
\draw (50pt,30pt)--(50pt,-30pt);
\draw[dotted] (-30pt,10pt)--(30pt,10pt);
\fill (-100pt,-10pt) circle (1.5pt) node[anchor=west] {$\Lambda_{(1)}$};
\fill (-50pt,-10pt) circle (1.5pt) node[anchor=west] {$\Lambda_{(2)}$};
\fill (50pt,-10pt) circle (1.5pt) node[anchor=west] {$\Lambda_{(p-1)}$};
\fill (100pt,-10pt) circle (1.5pt) node[anchor=west] {$\Lambda_{(p)}$};
\end{scope}
\end{tikzpicture}
\]
\end{lemma}

\begin{lemma}\label{Henningslemma-2}
  \[
\begin{tikzpicture}[scale=0.9]\footnotesize
\draw (-80pt,30pt)--(-80pt,-30pt);
\draw (-40pt,30pt)--(-40pt,-30pt);
\draw (80pt,30pt)--(80pt,-30pt);
\draw (40pt,30pt)--(40pt,-30pt);
\draw[color=white,line width=2mm] (-100pt,10pt)--(100pt,10pt);
\draw (-100pt,10pt)--(-25pt,10pt);
\draw[dotted] (-25pt,10pt)--(25pt,10pt);
\draw (25pt,10pt)--(100pt,10pt);
\fill (-80pt,-20pt) circle (1.5pt) node[anchor=west] {$\Lambda_{(1)}$};
\fill (-40pt,-20pt) circle (1.5pt) node[anchor=west] {$\Lambda_{(2)}$};
\fill (40pt,-20pt) circle (1.5pt) node[anchor=west] {$\Lambda_{(p-1)}$};
\fill (80pt,-20pt) circle (1.5pt) node[anchor=west] {$\Lambda_{(p)}$};
\node (=) at (4,0.2) {$=$};
\begin{scope}[xshift = 8cm]
\draw (-80pt,30pt)--(-80pt,-30pt);
\draw (-40pt,30pt)--(-40pt,-30pt);
\draw (80pt,30pt)--(80pt,-30pt);
\draw (40pt,30pt)--(40pt,-30pt);
\draw (-100pt,10pt)--(-25pt,10pt);
\draw[dotted] (-25pt,10pt)--(25pt,10pt);
\draw (25pt,10pt)--(100pt,10pt);
\fill (-80pt,-20pt) circle (1.5pt) node[anchor=west] {$\Lambda_{(1)}$};
\fill (-40pt,-20pt) circle (1.5pt) node[anchor=west] {$\Lambda_{(2)}$};
\fill (40pt,-20pt) circle (1.5pt) node[anchor=west] {$\Lambda_{(p-1)}$};
\fill (80pt,-20pt) circle (1.5pt) node[anchor=west] {$\Lambda_{(p)}$};
\end{scope}
\end{tikzpicture}
  \]
\end{lemma}

To write down the Hennigs invariant, we push all labeled coproduct factors along the upper circles to the base points. Then we multiply them and evaluate the resulting product by integrals. 

\begin{example}
	For Poincare homology 3-sphere $S$, the cointegral decorated surgery diagram of $S\#\overline{S}$ is shown in Fig.~\ref{fig:genus two} (ignore the vector field).  
	
	\begin{align*}Z_{Henn}(S\#\overline{S},H)=\sum\limits_{(\Lambda)}&\lambda\left(S(\Lambda'_{(3)})\Lambda_{(1)}\Lambda'_{(5)}S^{-1}(\Lambda_{(3)})\Lambda'_{(1)}S^{-1}(\Lambda_{(5)})\right)\\
	&\cdot\lambda\left(\Lambda_{(4)}S(\Lambda'_{(4)})S(\Lambda'_{(2)})\Lambda_{(2)}\Lambda'_{(6)}\Lambda'_{(7)}\Lambda'_{(8)}\right)
	\end{align*}
\end{example}

\subsection{Proof of Theorem} By Lemma \ref{Kuperberglemma-1}, \ref{Henningslemma-1} and \ref{Henningslemma-2}, we see that both $Z_{Kup}(M,f,H)$ and $Z_{Henn}(M\#\overline{M},H)$ can be written down by the same algorithm. That is one first creates coproduct factors of cointegrals along the lower circles and labels them at the intersections of the lower and upper circles, then multiplies the factors along upper circles following the same rule of antipode action and evaluates the product by integrals in the end. Therefore, $$Z_{Kup}(M,f,H)=Z_{Henn}(M\#\overline{M},H).$$

\begin{corollary}
Let $H$ be a Hopf algebra as in the theorem equipped with  anti-linear automorphism $\tau$ such that $(\tau\otimes\tau)(R)=R^{-1}_{21}$, then $Z_{Kup}(M,f,H)=|Z_{Henn}(M,H)|^2$.
\end{corollary}

\begin{proof}

When equipped with such anti-linear automorphism $\tau$ and the canonical ribbon element $\theta=u$, by Proposition 6.2. in \cite{H},  $Z_{Henn}(\overline{M},H)=\overline{Z_{Henn}(M,H)}$, Hence, we have
$$Z_{Kup}(M,f,H)=Z_{Henn}(M\#\overline{M},H)=Z_{Henn}(M,H)Z_{Henn}(\overline{M},H)=|Z_{Henn}(M,H)|^2.$$

\end{proof}

\begin{figure}\centering
  \begin{tikzpicture}[scale=0.6, fill=white]\tiny
  \draw [thick, name path=circle-1] (0,0) circle(1);
  \draw [thick, name path=circle-2] (13,0) circle(1);
  \draw [thick, name path=circle-3] (0,-10) circle(1);
  \draw [thick, name path=circle-4] (13,-10) circle(1);
\draw [blue, thick] (2.5,0) to [out=90, in=0] (0,3) to [out=180, in=90] (-7,-7) to [out=-90, in=180] (0,-17) to [out=0, in=-90] (6.9,-10) to [out=90, in=-165] (12.5,-5) to [out=15, in=-90] (18,0.5) to [out=90, in=0] (12,5) to [out=180, in=90] (6,0) to [out=-90, in=40] (2,-5) to [out=-140, in=90] (-3,-10) to [out=-90, in=180] (0,-13) to [out=0, in=-90] (2.5,-10) to [out=90, in=-140] (5.5,-5) to [out=40, in=-90] (10,0) to [out=90, in=180] (13,3) to [out=0, in=90] (16,0) to [out=-90, in=30] (8.5,-5) to [out=-150, in=90] (4.7,-10) to [out=-90, in=0] (0,-15) to [out=180, in=-90] (-5,-10) to [out=90, in=-140] (-1,-5) to [out=40, in=-90] (2.5,0); 
\draw [red, thick] (0,0) to [out=180, in=90] (-6.2,-8) to [out=-90, in=180] (0,-16) to [out=0, in=-90] (5.8,-10) to [out=90, in=-160] (10.3,-5) to [out=20, in=-90] (17,0.5) to [out=90, in=0] (12.5,4) to [out=180, in=90] (7.5,0) to [out=-90, in=40] (4,-5) to [out=-140, in=150] (-1,-10);
\draw [red, thick] (14,-10) to [out=40, in=0] (12.5,-8) to [out=180, in=90] (10,-10) to [out=-90, in=180] (12.5,-12) to [out=0, in=-90] (15.5,-9.5) to [out=90, in=0] (12.5,-7) to [out=180, in=90] (9,-10) to [out=-90, in=180] (12.5,-13) to [out=0, in=-90] (16.5,-9.5) to [out=90, in=0] (12.5,-6) to [out=180, in=90] (7.95,-10) to [out=-90, in=180] (12.5,-14) to [out=0, in=-90] (19.3,-0.5) to [out=90, in=0] (12,6) to [out=180, in=90] (4.5,0) to [out=-90, in=40] (0.5,-5) to [out=-140, in=90] (-4,-10) to [out=-90, in=180] (0,-14) to [out=0, in=-90] (3.6,-10) to [out=90, in=-140] (7,-5) to [out=40, in=-40] (14,0);
\draw [green, thick, rounded corners=5pt] (1,0)--(9.7,0)--(9.7,-0.4)--(10.2,-0.4)--(10.3,0)--(12,0);
\draw [green, thick, rounded corners=5pt] (1,-10)--(9.7,-10)--(9.7,-10.4)--(10.2,-10.4)--(10.3,-10)--(12,-10);
  \draw [dashed,rounded corners=10pt] (2.5,7.5)--(2.5,0)--(2.3,-1);
  \draw [dashed] plot[smooth ,tension=1] coordinates{(2.3,-1) (0.4,-2.5) (-3.5,-3) (-7, -3)};
  \draw [->,line width=0.25mm] (2.5,3.5)--(2.5,3.4);
  \draw [->,line width=0.25mm] (1.3,-2.1)--(1.5,-2);
  \draw [->,line width=0.25mm] (-3.1,-3)--(-3,-3);
  \draw [dashed] plot[smooth ,tension=1] coordinates{(0,0) (1.5,1.5) (2, 7.5)};
  \draw [->,line width=0.25mm] (1.8,3.6)--(1.8,3.5);
  \draw [dashed] plot[smooth ,tension=1] coordinates{(0,0) (-2,3) (-7, 6)};
  \draw [->,line width=0.25mm] (-2.6,3.5)--(-2.5,3.4);
  \draw [dashed] plot[smooth ,tension=1] coordinates{(0,0) (-2,0.5) (-7, 1)};
  \draw [->,line width=0.25mm] (-3.8,0.7)--(-3.7,0.7);
  \draw [dashed] plot[smooth ,tension=1] coordinates{(0,0) (1.2,-1.1) (-2,-2) (-7, -2)};
  \draw [->,line width=0.25mm] (1.25,-0.7)--(1.15,-0.6);
  \draw [->,line width=0.25mm] (-3.7,-2.1)--(-3.6,-2.1);
  \draw [dashed,rounded corners=10pt] (10,7.5)--(10,0)--(10.2,-1);
  \draw [dashed] plot[smooth ,tension=1] coordinates{(10.2,-1) (12.6,-2.5) (16.5,-3) (20, -3)};
  \draw [->,line width=0.25mm] (10,3.4)--(10,3.5);
  \draw [->,line width=0.25mm] (11.5,-2)--(11.7,-2.1);
  \draw [->,line width=0.25mm] (16,-3)--(16.1,-3);
  \draw [dashed] plot[smooth ,tension=1] coordinates{(13,0) (11.5,1.5) (10.5, 7.5)};
  \draw [->,line width=0.25mm] (11,3.5)--(11,3.6);
  \draw [dashed] plot[smooth ,tension=1] coordinates{(13,0) (15,3) (20, 6)};
  \draw [->,line width=0.25mm] (15.5,3.4)--(15.6,3.5);
  \draw [dashed] plot[smooth ,tension=1] coordinates{(13,0) (15,0.5) (20, 1)};
  \draw [->,line width=0.25mm] (16.7,0.7)--(16.8,0.7);
  \draw [dashed] plot[smooth ,tension=1] coordinates{(13,0) (11.8,-1.1) (15,-2) (20, -2)};
  \draw [->,line width=0.25mm] (11.85,-0.6)--(11.75,-0.7);
  \draw [->,line width=0.25mm] (16.6,-2.1)--(16.7,-2.1);
  \draw [dashed] plot[smooth ,tension=1] coordinates{(3,7.5) (4,3) (6.5,1) (8.5,3) (9.5,7.5)};
  \draw [->, dashed, line width=0.25mm] (3.7,3.6)--(3.8,3.4);
  \draw [->, dashed, line width=0.25mm] (6.6,1)--(6.7,1);
  \draw [->, dashed, line width=0.25mm] (8.6,3.3)--(8.7,3.6);
  \draw [dashed] (0,0)--(13,0);
  \draw [->, dashed, line width=0.25mm] (1.5,0)--(1.4,0);
  \draw [->, dashed, line width=0.25mm] (3.4,0)--(3.5,0);
  \draw [->, dashed, line width=0.25mm] (6.5,0)--(6.6,0);
  \draw [->, dashed, line width=0.25mm] (9,0)--(9.1,0);
  \draw [->, dashed, line width=0.25mm] (11,0)--(10.9,0);
  \draw [dashed] plot[smooth ,tension=1] coordinates{(-7, -4) (0,-3.5) (6.5,-1.5) (13,-3.5) (20, -4)};
  \draw [->,line width=0.25mm] (-3.0,-3.9)--(-2.9,-3.9);
  \draw [->,line width=0.25mm] (0.4,-3.4)--(0.5,-3.38);
  \draw [->,line width=0.25mm] (6.5,-1.5)--(6.6,-1.5);
  \draw [->,line width=0.25mm] (11.7,-3.2)--(11.8,-3.25);
  \draw [->,line width=0.25mm] (15.9,-3.9)--(16.0,-3.9);
  \draw [dashed] (-7,-5)--(20,-5);
  \draw [->,line width=0.25mm] (-3,-5)--(-2.9,-5);
  \draw [->,line width=0.25mm] (1.6,-5)--(1.7,-5);
  \draw [->,line width=0.25mm] (6.5,-5)--(6.6,-5);
  \draw [->,line width=0.25mm] (11.4,-5)--(11.5,-5);
  \draw [->,line width=0.25mm] (15.9,-5)--(16,-5);
  \draw [dashed] plot[smooth ,tension=1] coordinates{(-7, -6) (0,-6.5) (6.5,-8.5) (13,-6.5) (20, -6)};
  \draw [->, dashed, line width=0.25mm] (3.7,-13.6)--(3.8,-13.4);
  \draw [->, dashed, line width=0.25mm] (6.6,-11)--(6.7,-11);
  \draw [->, dashed, line width=0.25mm] (8.6,-13.3)--(8.7,-13.6);
  \draw [dashed] (0,-10)--(13,-10);
  \draw [->, dashed, line width=0.25mm] (1.5,-10)--(1.4,-10);
  \draw [->, dashed, line width=0.25mm] (3.4,-10)--(3.5,-10);
  \draw [->, dashed, line width=0.25mm] (6.5,-10)--(6.6,-10);
  \draw [->, dashed, line width=0.25mm] (8.6,-10)--(8.7,-10);
  \draw [->, dashed, line width=0.25mm] (11,-10)--(10.9,-10);
  \draw [dashed] plot[smooth ,tension=1] coordinates{(3,-17.5) (4,-13) (6.5,-11) (8.5,-13) (9.5,-17.5)};
  \draw [->,line width=0.25mm] (-3.0,-6.1)--(-2.9,-6.1);
  \draw [->,line width=0.25mm] (0.9,-6.7)--(1,-6.75);
  \draw [->,line width=0.25mm] (6.5,-8.5)--(6.6,-8.5);
  \draw [->,line width=0.25mm] (15.9,-6.1)--(16.0,-6.1);
  \draw [->,line width=0.25mm] (11.7,-6.8)--(11.8,-6.75);
  \draw [dashed,rounded corners=10pt] (2.5,-17.5)--(2.5,-10)--(2.3,-9);
  \draw [->,line width=0.25mm] (2.5,-13.5)--(2.5,-13.4);
  \draw [->,line width=0.25mm] (1.3,-7.9)--(1.5,-8);
  \draw [->,line width=0.25mm] (-3.1,-7)--(-3,-7);
  \draw [dashed] plot[smooth ,tension=1] coordinates{(2.3,-9) (0.4,-7.5) (-3.5,-7) (-7, -7)};
  \draw [dashed] plot[smooth ,tension=1] coordinates{(0,-10) (1.5,-11.5) (2, -17.5)};
  \draw [->,line width=0.25mm] (1.8,-13.6)--(1.8,-13.5);
  \draw [dashed] plot[smooth ,tension=1] coordinates{(0,-10) (-2,-13) (-7, -16)};
  \draw [->,line width=0.25mm] (-2.6,-13.5)--(-2.5,-13.4);
  \draw [dashed] plot[smooth ,tension=1] coordinates{(0,-10) (-2,-10.5) (-7, -11)};
  \draw [->,line width=0.25mm] (-3.8,-10.7)--(-3.7,-10.7);
  \draw [dashed] plot[smooth ,tension=1] coordinates{(0,-10) (1.2,-8.9) (-2,-8) (-7, -8)};
  \draw [->,line width=0.25mm] (1.25,-9.3)--(1.15,-9.4);
  \draw [->,line width=0.25mm] (-3.7,-7.9)--(-3.6,-7.9);
  \draw [dashed,rounded corners=10pt] (10,-17.5)--(10,-10)--(10.2,-9);
  \draw [dashed] plot[smooth ,tension=1] coordinates{(10.2,-9) (12.6,-7.5) (16.5,-7) (20, -7)};
  \draw [->,line width=0.25mm] (10,-13.4)--(10,-13.5);
  \draw [->,line width=0.25mm] (11.5,-8)--(11.7,-7.9);
  \draw [->,line width=0.25mm] (16,-7)--(16.1,-7);
  \draw [dashed] plot[smooth ,tension=1] coordinates{(13,-10) (11.5,-11.5) (10.5, -17.5)};
  \draw [->,line width=0.25mm] (11,-13.5)--(11,-13.6);
  \draw [dashed] plot[smooth ,tension=1] coordinates{(13,-10) (15,-13) (20, -16)};
  \draw [->,line width=0.25mm] (15.5,-13.4)--(15.6,-13.5);
  \draw [dashed] plot[smooth ,tension=1] coordinates{(13,-10) (15,-10.5) (20, -11)};
  \draw [->,line width=0.25mm] (16.7,-10.7)--(16.8,-10.7);
  \draw [dashed] plot[smooth ,tension=1] coordinates{(13,-10) (11.8,-8.9) (15,-8) (20, -8)};
  \draw [->,line width=0.25mm] (11.85,-9.4)--(11.75,-9.3);
  \draw [->,line width=0.25mm] (16.6,-7.9)--(16.7,-7.9);
  \fill [white] (0,0) circle(0.95);
  \fill [white] (13,0) circle(0.95);
  \fill [white] (0,-10) circle(0.95);
  \fill [white] (13,-10) circle(0.95);
  \draw (2.4,0) node[anchor=north west] {$\Lambda_{(1)}$}; 
  \draw (4.4,0) node[anchor=north west] {$\Lambda_{(2)}$};
  \draw (5.9,0) node[anchor=north west] {$\Lambda_{(3)}$}; 
  \draw (7.4,0) node[anchor=north west] {$\Lambda_{(4)}$};
  \draw (10,-0.1) node[anchor=north west] {$\Lambda_{(5)}$}; 
  \draw (2.4,-10) node[anchor=north west] {$\Lambda'_{(1)}$}; 
  \draw (3.5,-10) node[anchor=north west] {$\Lambda'_{(2)}$};
  \draw (4.6,-10) node[anchor=north west] {$\Lambda'_{(3)}$}; 
  \draw (5.7,-10) node[anchor=north west] {$\Lambda'_{(4)}$};
  \draw (6.8,-10) node[anchor=north west] {$\Lambda'_{(5)}$}; 
  \draw (7.85,-10) node[anchor=north west] {$\Lambda'_{(6)}$}; 
  \draw (8.9,-10) node[anchor=north west] {$\Lambda'_{(7)}$}; 
  \draw (10,-10.1) node[anchor=north west] {$\Lambda'_{(8)}$}; 
  \end{tikzpicture}
\caption{}
\label{fig:genus two}
\end{figure}

\end{document}